\documentclass[12pt,letterpaper]{amsart}
\usepackage[makeindex]{imakeidx}
\usepackage{ amsmath,amsthm, amsbsy,amsfonts,amssymb, txfonts}
\usepackage[normalem]{ulem}
\usepackage{stmaryrd,mathrsfs,graphicx, supertabular, amscd, tikz-cd, tikz}
\usepackage{stackrel}
\usetikzlibrary{matrix,arrows,decorations.pathmorphing}

\DeclareFontEncoding{LS1}{}{}
\DeclareFontSubstitution{LS1}{stix}{m}{n}
\DeclareSymbolFont{symbols2}{LS1}{stixfrak} {m} {n}
\DeclareMathSymbol{\operp}{\mathbin}{symbols2}{"A8}

\usepackage[active]{srcltx}
\allowdisplaybreaks
\numberwithin{equation}{section}

\usepackage[
	hypertexnames=false,
	hyperindex,
	pagebackref,
	breaklinks=true,
	bookmarks=false,
	colorlinks,
	linkcolor=blue,
	citecolor=red,
	urlcolor=red,
]{hyperref}
\usepackage{hyperref}

\def\CC{{\mathbb C}}

\def\FF{{\mathbb F}}

\def\PP{{\mathbb P}}

\def\QQ{{\mathbb Q}}

\def\ZZ{{\mathbb Z}}

\def\ssm{\smallsetminus}

\def\g{\gamma}

\def\int{{\rm int}}

\newcommand{\eps}{\varepsilon}

\newcommand{\p}{\partial}

\def\la{\langle}
\def\ra{\rangle}
\def\half{{\tfrac{1}{2}}}

\def\pt{{\scriptscriptstyle\bullet}}

\newcommand\aut{\operatorname{Aut}}

\newcommand\coker{\operatorname{coker}}

\newtheorem{theorem}{Theorem}[section]
\newtheorem{lemma}[theorem]{Lemma}

\newtheorem{corollary}[theorem]{Corollary}

\theoremstyle{remark}
\newtheorem{example}[theorem]{Example}

\newtheorem{remark}[theorem]{Remark}

\title[Level structures on hypersurfaces]{Level structures on cyclic covers of 
$\PP^n$ and the homology of Fermat hypersurfaces}
\begin{document}
\author{Eduard Looijenga}
\address{Department of Mathematics, University of Chicago (USA) and Mathematisch Instituut, Universiteit Utrecht (Nederland)}
\email{e.j.n.looijenga@uu.nl}
\maketitle

\begin{abstract}
Let  $Z'\subset \PP^{n}$ be a smooth projective hypersurface of degree $d>1$ and let $Z\to \PP^n$ be the $\mu_d$-cover totally ramified  along $Z'$. We relate full level $d$ structures on the primitive cohomology 
$Z'$  with full level $d$ structures on the primitive cohomology of $Z$. In the special case, $d=n=3$
this makes a  marking of a smooth cubic surface determine  a  level $3$-structure on the associated 
cubic threefold, thereby answering a question by Beauville \cite{beauville}. We expect many more such applications.
\end{abstract}

\section*{Introduction}

Let  $Z'\subset \PP^{n}$ be a smooth projective hypersurface of degree $d>1$ and let $Z\to \PP^n$ be the $\mu_d$-cover totally ramified  along $Z'$. Our main theorem \ref{thm:main} states that if we take the $\mu_d$-co-invariants of the  primitive  homology of $Z$ and reduce modulo $d$, then this naturally surjects onto the image of the primitive homology of $Z'$ on the primitive cohomology of $Z'$, followed by reduction modulo $d$. When $d$ is a  prime, we can also  make explicit what the kernel of this surjection is (Corollary \ref{cor:main}). The case when $Z'$ is a cubic surface of particular interest because of the way this relates to the ball quotient description of its moduli space by Allcock-Carlson-Toledo \cite{act}.
 
The proof is based on two simple maps we associate with a module $M$ for the group ring of any finite group $G$ (as defined in Section \ref{sect:group rings}). We apply this to a particular $\ZZ\mu_d$-module $M$ whose invariants  $M^{\mu_d}$  yield the primitive cohomology of $Z'$ and for which $M/M^{\mu_d}$  gives the 
 primitive homology of $Z$. We also apply this to the dual of $M$ for which the dual assertion holds.
We expect this to have similar applications to other instances of a smooth $G$-coverings relating the homology of the cover with that of the branch locus (which coefficients dividing order of the abelianization of $G$). 

We expect our main result to generalize  with other applications of interest. We justify this  in Remark \ref{rem:incarnations}, where we briefly discuss  the case of  a del Pezzo surface of degree 2 as double a cover of $\PP^2$,  a K3 surface of degree 4 as a double cover of del Pezzo surface of degree 2 and  a K3 surface of degree  6 as a  $\mu_3$-cover of  $\PP^1\times\PP^1$. These last two covers were used by Kond\=o for his ball quotient description of the moduli space of genus 3 resp. genus 4 curves. We thus find that these ball quotient descriptions subsist  we impose on a level 2 structure on a genus 3 curve or a level 3 structure on a genus 4 curve. 

Our main result also leads to a precise description of the (co)primitive cohomology of the smooth Fermat hypersurface as a module over its symmetry group (and thus correct an error in \cite{looijenga}). 
Dami\'an Gvirtz alerted me (back in March 2021) to this (which he and Skorobogatov corrected in the surface case in \cite{gvirtz}) and this was later also observed  by Nicholas Addington. 

The situation discussed here belongs to a classical theme in algebraic topology. Here the conventional  approach is  for a space $X$ on which a finite  group $G$ acts, to compare  the equivariant cohomology  $H_G^\pt(X)$  and  the cohomology of its fixed point set $H^\pt(X^G)$. It would nice to see  our main result expressed in these terms and obtained along these lines.  
\smallskip

I  am grateful to Dami\'an Gvirtz and Nicholas Addington for  pointing out the issue mentioned above. 
I am further  indebted to Nicholas  and Benson Farb for their comments on an earlier draft.

\section{Classical Lefschetz theory with integral coefficients}\label{} Let $Z\subset \PP^{n+1}$ be a smooth hypersurface of degree $d>1$. 
It is well-known that its integral homology and cohomology is torsion free (the discussion below will reprove this).
Let $\eta\in H^2(\PP^n)$ denote the hyperplane class and 
$\eta_Z\in H^2(Z)$ its restriction to $Z$. Then  $\eta_ Z^n$ takes on the fundamental class $[Z]\in H_{2n}(Z)$ the value $d$.
The \emph{primitive cohomology} of $Z$, denoted here by $H^P(Z)$,   is the cokernel of the natural map $H^{n} (\PP^{n+1})\to H^{n} (Z)$ and dually,  its \emph{primitive homology} $H_P(Z)$ as the kernel  of  
$H_{n} (Z)\to H_n(\PP^{n+1})$. Both  are  torsion free and each other's dual. The  composite 
\[
\iota_Z: H_P(Z)\hookrightarrow H_{n} (Z)\cong H^{n} (Z)\twoheadrightarrow H^P(Z),
\]
(where the middle map is given by Poincar\'e duality) defines the intersection pairing restricted to $H_P(Z)$.
All these maps are isomorphisms when $n$ is odd, whereas for $n$ is even the composite is injective with cokernel canonically isomorphic with $\ZZ/d$.
Note that when $n=0$ (in which case $Z$ is a $d$-element subset of $\PP^1$), we have identifications  
$H_P(Z)\cong \tilde H_0(Z)$ (the linear combinations of the points on $Z$ with zero coefficient sum)
and  $H^P(Z)=\tilde H^0(Z)$ (the linear combinations of the points on $Z$ modulo their sum). 

We have a perfect pairing between  $H^P(Z)$ and $H_P(Z)$, denoted 
 $(\alpha,b)\in H^P(Z)\times H_P(Z)\mapsto (\alpha|b)$. This enables us identify one with the other's dual.
The intersection pairing on $H^P(Z)$ satisfies $a\cdot b= (\iota_Z(a)|b)$.

The following lemma must be well-known.

\begin{lemma}\label{lemma:hyperplanecut}
Let $Z'\subset Z$ be  smooth hyperplane section and put $\mathring Z:=Z\ssm Z'$. Then  $\mathring Z$ is a smooth affine hypersurface which has the homotopy type of a wedge of $(d-1)^{n+1}$ $n$-spheres and we have a short exact sequences fitting in the (self-dual) commutative diagram 
\begin{center}
\begin{tikzcd}
0\rar &H^P(Z')\rar  & \tilde H_n(\mathring Z)\rar\arrow[d, "\iota_{\mathring Z}"] &  H_P(Z)\rar\arrow[d, "\iota_{Z}"]  & 0\\
0 & H_P(Z')\lar &  \tilde H^n(\mathring Z)\lar & H^P(Z)\lar & 0\lar\\
\end{tikzcd}
\end{center}
of which the vertical maps are the natural ones (defining  the intersection pairing on $\tilde H_n(\mathring Z)$). The precomposition  of  $\tilde H_n(\mathring Z)\to \tilde H^n(\mathring Z)$ with  $H^P(Z')\to \tilde H_n(\mathring Z)$  and its  postcomposition  with $\tilde H^n(\mathring Z)\to H_P(Z')$ are both zero.
\end{lemma}
\begin{proof}
Since $\mathring Z$ is a Milnor fiber  of the affine  cone over $Z'$, its homotopy type is that of a wedge of $(d-1)^{n+1}$ $n$-spheres whose number is equal to the Milnor number, which is $(d-1)^{n+1}$ (\cite{milnor}, Thm.\ 6.5).

We  next establish the lower exact sequence. 
When $n=0$, this boils down the observation we just made, namely that $H^P(Z)\cong \tilde H^0(\mathring Z)$.
We therefore assume $n>0$. Then  
$\tilde H^\pt(\mathring Z)$ is concentrated in degree $n$ and hence the Gysin sequence of the pair $(Z,Z')$ gives the exact sequence
\[
0\to H^{n-2}(Z')\to H^n(Z)\to H^n(\mathring Z)\to H^{n-1}(Z')\to  H^{n+1}(Z)\to 0.
\]

If $n=2m$ is even with $m>0$, then $H^{n+1}(Z)=0$. Furthermore, $H^{n-2}(Z')$ is freely generated by $\eta_{Z'}^{m-1}$ and  maps to $\eta_Z^{m}\in  H^n(Z)$.

If $n=2m-1$ is odd, then $H^{n-2}(Z')=0$  and  $H^{n+1}(Z)$ is freely generated by $\eta_Z^m$ and again the assertion follows.

The proof of the upper exact sequence  is similar. The last assertion is seen geometrically: it is induced by the map
\[
H_{n-1}(Z')\xrightarrow{L}H_n(\mathring Z)\to H^n(\mathring Z)\xrightarrow{R} H^{n-1}(Z'), 
\]
where $L$ is   the Lefschetz tube mapping: it assigns to an $(n-1)$-cycle on $Z'$ its preimage in a tubular neighbourhood boundary of $Z'$ in $Z$ (which can indeed be understood as an $n$-cycle) and $R$ 
the residue map (which is dual to $L$). The composite of any of these with the middle map  is zero. Indeed, the above sequence is self-dual.
\end{proof}

\begin{remark}\label{rem:vanlattice}
The group $H_P(Z')$ is in fact a \emph{vanishing lattice}. We recall what this means and how this comes about.
Let  $\PP_{n,d}$  stand for the complete linear system of degree $d$  hypersurfaces in $\PP^n$. The singular 
hypersurfaces are parametrized by the  \emph{discriminant hypersurface} $D_{n,d}\subset \PP_{n,d}$, which is irreducible and  whose smooth points define hypersurfaces with a single ordinary double point and no other singularities. We put $ \mathring\PP_{n,d}:=\PP_{n,d}\ssm D_{n,d}$. 
A path $\g: [0,1]\to \PP_{n,d}$ with $\g(0)=[Z']$ and which hits $D_{n,d}$ at its end point for the first time  and does so in a smooth point $[Z'_1]$ of   $D_{n,d}$,  determines  a \emph{vanishing cycle}  $\delta_\g\in H_P(Z')$. The path $\g$ also determines up to homotopy  a loop in $\mathring\PP_{n,d}$ based at $[Z']$  circling around $[Z'_1]$ and its 
monodromy action on $H_P(Z')$ is expressed in terms $\delta_\g$ by the Picard-Lefschetz formula 
\[
T(\delta_\g): a\in H_P(Z')\mapsto a +(-1)^{n(n-1)/2} (a\cdot \delta_\g)\delta_\g.
\]
In case $n=2m+1$ is odd, the intersection pairing on $H_P(Z')$ is symmetric,  
$\delta_\g\cdot\delta_g=(-1)^m 2$  and hence  $T(\delta_\g)$ is an orthogonal reflection. For even $n$ it is a symplectic transvection.
The  loops thus obtained  make up a single generating  conjugacy class  in $\pi_1(\mathring\PP_{n,d}, [Z'])$ and the classical Lefschetz theory asserts that the set $\Delta_{Z'}\subset H_P(Z')$ of the associated  vanishing cycles generates $H_P(Z')$. So the 
associated Picard-Lefschetz transformations $T(\delta)$ make up a generating conjugacy class of  the image of the monodromy representation $\rho_{Z'}: \pi_1(\mathring\PP_{n,d}, [Z'])\to \aut H_P(Z')$ and this image has $\Delta_{Z'}$ as an  orbit.  The preceding also implies  that for odd $n$,  the intersection pairing on $H_P(Z')$ is  even  and hence lifts to quadratic form  $H_P(Z')\to \ZZ$ which  takes the value  $(-1)^{(n-1)(n-2)/2}$ on the vanishing cycles. For even $n$, the preceding still shows that the  reduction modulo 2 of the intersection pairing  lifts to  quadratic form $H_P(Z'; \FF_2)\to \FF_2$ which takes the value $1$ on the  images of the vanishing cycles.
\end{remark}

It is clear that primitive homology and primitive cohomology have the same rank. Lemma \ref{lemma:hyperplanecut} implies  the (known) formula for that common rank:

\begin{corollary}\label{cor:}
The rank $p_n(d)$ of the (co)primitive cohomology of a smooth  degree $d$ hypersurface  in $\PP^{n+1}$ is
\[
(d-1)\frac{(d-1)^{n+1}+(-1)^n}{d}. 
\]
\end{corollary}
\begin{proof} First observe that $p_0(d)=d-1$. The corollary then follows with induction, for  Lemma \ref{lemma:hyperplanecut} gives the recurrence  relation $p_n(d)=(d-1)^{n+1}+p_{n-1}(d)$.
\end{proof}

Note the identity 
\begin{equation}\label{eqn:funny}
p_n(d)=(d-1)\big(p_{n-1}(d)+(-1)^n\big).
\end{equation}

The intersection pairing 
$(\alpha, \beta)\in H^{n}(Z)\times H^{n}(Z)\to \alpha\cdot\beta:=\la \alpha \cup\beta | [Z]\ra\in \ZZ$
is perfect by Poincar\'e duality.  It is symmetric or antisymmetric  according to whether $n$ is even or odd. 

Suppose  $n=2m$ is even. Then  $\eta^m_Z\in H^n(Z)$ is indivisible and
has self-intersection $d$. Since we may identify $H_P(Z)$ with  the orthogonal complement of 
$\eta^m_Z$, it follows that  $H_P(Z)\to H^P(Z)$ is an injection with cokernel  $\ZZ/d$. We regard this injection as an embedding. So then $H_P(Z)$ contains $d H^P(Z)$ and $H_P(Z)/d H^P(Z)$
is a  free $\ZZ/d$-submodule of $H^P(Z; \ZZ/d)$ of corank one.

\begin{corollary}\label{cor:complex}
Let $Z=Z^0\supset Z^1\supset\cdots Z^n\supset Z^{n+1}=\emptyset$ be the transversal intersection of $Z$ with a complete flag in $\PP^{n+1}$,  yielding the affine stratification with codimension $k$ stratum
$\mathring Z^k:= Z^k\ssm Z^{k+1}$. Then the  maps 
\[
D_k: \tilde H_{n-k}( \mathring Z^{k})\twoheadrightarrow H_P(Z^{k})\xrightarrow{\iota_{Z^k}} H^P(Z^{k})\hookrightarrow
\tilde H_{n-k-1}( \mathring Z^{1+k})
\]
define a (homological) complex  which naturally maps to the single term complex  $H_P(Z)$ placed in degree zero. The kernel of this chain map has its homology in degrees $0, -1, \dots , 1-n$ and is in these degrees $\ZZ/d$. \hfill $\square$
\end{corollary}

\section{A pair of  maps associated with a module over a group ring}\label{sect:group rings} 

We shall apply what follows to a finite  cyclic group, but since  the discussion initially does not require that assumption, 
we let $G$ be any group of finite order $d$. The  group ring $\ZZ G$  comes with an
\emph{augmentation} resp.\ \emph{co-augmentation}
\begin{gather*}
\eps_G: \ZZ G\to \ZZ, \quad  \eps_G(g)=1 \text{  for all $g\in G$},\\
\textstyle \eps^G:  \ZZ\to \ZZ G,  \quad \eps^G(1)=\sum_{g\in G} g.
\end{gather*}
These are $\ZZ G$-linear if we regard $\ZZ$ as the trivial $\ZZ G$-module (we denoted  the  image  of $g\in G$ in $\ZZ G$ also by  $g$).  They form respectively the invariants resp.\ co-invariants of $\ZZ G$ as a (left) module over itself.
Then $\ker(\eps_G)$ and the image of $\eps^G$ are each others annihilator and   $\ker(\eps_G)\cong \coker(\eps^G)$ as $\ZZ G$-modules. Note that $\eps_G\eps^G: \ZZ\to \ZZ$  is multiplication with $d$.
We shall write   $I$ for $\ker(\eps_G)$ (the \emph{augmentation ideal}).

Let $M$ be a $\ZZ G$-module. The map
\[
\textstyle \eps^G_M: M\to M^G, m\mapsto \eps^G(1)m=\sum_{g\in G} gm
\]
has the property that  $\eps^G_M (gm)= \eps^G_M(m)$ and is on $M^G$ multiplication by $d$
Hence it induces a $\ZZ/d$-linear map
\begin{equation}\label{eqn:natural}
r_M:  (M/M^G)_G\to \ZZ/d\otimes_\ZZ M^G.
\end{equation}

For example,   $r_{\ZZ G}$ is an isomorphism between two copies of $\ZZ/d$ and 
$r_\ZZ$ is the inclusion of $0$ in $\ZZ/d$.
If $M=I$, then $I^G=0$ and so the target of $r_{I}$ is zero  and its domain is  
$(I)_G\cong I/I^2$. As is well-known, $I/I^2$ can be identified with (the additively written) abelianization $G_{ab}$ of $G$.

There is also a dual version based on the isomorphism $M/IM\cong M_G$:  the natural map 
$I\otimes_\ZZ M\to I M\to (I M)_G$ factors through 
$I/I^2\otimes_\ZZ M_G$ and thus defines a map
\[
s_M: G_{ab}\otimes_\ZZ M_G\to (I M)_G.
\]

\section{Cyclic covers }\label{sect:cyclic covers} 
In what follows we are concerned with the group $\mu_d\subset \CC^\times$ of $d$th roots of unity. We shall write
$\eps_d: \ZZ\mu_d\to \ZZ$ and $\eps^d:\ZZ\to \ZZ\mu_d$ for augmentation and co-augmentation.  
 If $g$ is a generator of $\mu_d$, then $g-1$ is a generator of $I_d$  and  
$\eps^d(1)=\sum_{g\in \ZZ/d} g^k$. We shall also write $R_d$ for $\ZZ\mu_d$ and $A_d$ for the `co-augmentation ring' $\ZZ\mu_d/\eps^{\mu_d}(\ZZ)$. Note that $I_d$ is a principal $A_d$-module.
Both are free $\ZZ$-modules of rank $d-1$.

\smallskip
Let $Z'\subset \PP^n$ be a smooth hypersurface of degree $d$ and let  $F\in \CC[z_0, \dots, z_n]$ be a defining equation (so homogeneous of degree $d$). Then $F+z_{n+1}^d$ defines  a smooth degree $d$ hypersurface $Z\subset \PP^{n+1}$ which comes with an action by $\mu_d$ defined by
\begin{multline*}
\zeta\in \mu_d: [z_0:\cdots :z_n:z_{n+1}]\mapsto [z_0:\cdots :z_n: \zeta z_{n+1}]=\\=
 [\zeta^{-1}z_0:\cdots :\zeta^{-1}z_n: z_{n+1}].
\end{multline*}
(It was recently shown by Zhiyuan Li and Zhichao Tang \cite{li_tang} that $Z$  determines $Z'$ up to projective equivalence.)

The $\mu_d$-orbit space of $Z$ can be identified with $\PP^n$ and the resulting map $\pi: Z\to \PP^n$ 
takes the  $\mu_d$-fixed point set in $Z$ isomorphically onto $Z'$. We therefore regard 
$Z'$ as a subvariety of $Z$,  given as the (transversal) hyperplane section defined by $z_{n+1}=0$. 
The diagram of Lemma \ref{lemma:hyperplanecut} is obviously one of $\ZZ\mu_d$-modules.
The intersection pairing on $H_P(Z)$ is $\mu_d$-invariant. This determines a sesquilinear
pairing 
\[
\textstyle \la a,b \ra:=\sum_{g\in\mu_d} (a\cdot g b)g\in R_d,
\]
which  is $R_d$-linear in the first variable. This defines in fact  a $(-)^n$-hermitian form on the $R_d$-module $H_P(Z)$, for if we denote the natural anti-involution of $R_d$  which takes $g$ to $g^{-1}$ by a  bar, then $\overline{\la a, b \ra}=(-1)^n\la b, a \ra$. 

\begin{lemma}\label{lemma:rX}
In the  upper exact sequence of Lemma \ref{lemma:hyperplanecut}, 
the trivial $\ZZ\mu_d$-module  $H^P(Z')$  maps isomorphically onto
$\tilde H_n(\mathring Z)^{\mu_d}$. Dually, in its lower exact sequence, $\tilde H^n(\mathring Z)_{\mu_d}$ maps 
isomorphically to $H_P(Z')$.
\end{lemma}
\begin{proof} 
Since $H_P(Z)$ is torsion free, it suffices to show that the trivial character of $\mu_d$ will not appear in 
$H_P(Z; \QQ)$. This follows from the fact that the invariant rational cohomology of $Z$ must come form its 
space of $\mu_d$-invariants, which is $\PP^n$, and the fact that the image of 
$H^n(\PP^n; \QQ)\to H^n(Z; \QQ)$ is supplemented by $\QQ\otimes H_P(Z)$.
\end{proof}

\begin{remark}\label{rem:}
We digress a bit to explicate the monodromy representation.

The $\mu_d$-covering construction cannot  be carried out universally over $\PP_{n,d}$ (this due to the fact that such a covering is only unique up to a covering transformation), but what we can do is to consider
the subspace  of  $\PP_{n+1,d}$ parametrizing  hypersurfaces of the form $F(z_0,\dots , z_n)+z_{n+1}^d$ with $F\not=0$. This amounts to a base change over the (tautological)  $\CC^\times$-bundle over  $\PP_{n,d}$. 
The fundamental group of its restriction to $\mathring\PP_{n,d}$ is a central extension 
 of $\pi_1(\mathring\PP_{n,d}, [Z'])$ by $\pi_1(\CC^\times, 1)\cong\ZZ$.  
Here $\pi_1(\CC^\times, 1)$ acts by covering transformations and gives a surjection $\pi_1(\CC^\times, 1)\to \mu_d$ (and thus determines a generator $g_0$ of $\mu_d$). So if $\hat\pi_1(\mathring\PP_{n,d}, [Z'])$ is  the corresponding central extension  of $\pi_1(\mathring\PP_{n,d}, [Z'])$ by $\mu_d$ (a push-out), then the 
monodromy representation   now factors through a homomorphism 
\[
\hat\rho_{Z'}:\hat\pi_1(\mathring\PP_{n,d}, [Z'])\to \aut H_P(Z).
\]
with the central $\mu_d\subset \hat\pi_1(\mathring\PP_{n,d}, [Z'])$ acting in the given manner.
This group acts in the contragradient manner on $H^P(Z)$  and  $\iota_Z: H_P(Z)\to H^P(Z)$ is  equivariant.
 Its image of this representation is  generated by `generalized Picard-Lefschetz transformations' in the sense of Pham:
a path $\g$ as above determines a degeneration into a $\mu_d$-cover $Z_1\to \PP^n$ totally ramifying along $Z'_1$. In terms of singularity theory, this gives us vanishing sublattice  $V_\g\subset H_P(Z)$ of type $A_{d-1}$.  Pham specifies  a vanishing cycle $\hat\delta_\g\in H_P(Z)$ (denoted $e$ in \S 2 of \cite{pham}), which generates $V_\g$ as a free $A_d$-module of rank one and shows that the  associated monodromy  is given by the transformation in $H_P(Z)$ defined by 
\[
\textstyle \hat T(\delta_{\g}): a\mapsto a + (-1)^{(n+1)n/2} \la \hat a,\delta_{\g}\ra \hat\delta_{\g}=a + (-1)^{(n+1)n/2} \sum_{g\in \mu_d}  (a, g\hat\delta_{\g}) g\hat\delta_{\g} 
\]
This transformation respects the form $\la\; ,\; \ra$; this makes it  for even $n$ a unitary reflection of order $d$. 
\end{remark}

Lemma \ref {lemma:rX} enables us to apply the discussion in Section \ref{sect:group rings} to $M=\tilde H_n(\mathring Z)$.  We find that multiplication with $\eps^d(1)=\sum_{g\in \mu_d} g$ in the $\ZZ\mu_d$-module $\tilde H_n(\mathring Z)$ determines a linear map 
\[
r_{Z'}:  H_P(Z)_{\mu_d}\to  \ZZ/d\otimes  H^P(Z').
\]
It is worth spelling out its  definition  in geometric terms: a class in $H_P(Z)$ can be represented by an $n$-cycle $a$ on $Z$ with support in $Z\ssm Z'$. Then the $\mu_d$-invariant $n$-cycle $\sum_{g\in \mu_d} ga$  bounds in $Z$: $\sum_{g\in \mu_d} ga=\p c$ for some $(n+1)$-chain $c$ on $Z$. Then  $c$ cuts out on 
$Z'$ an  $(n-1)$-cycle which represents $r_{Z'}[a]$.

If we apply Lemma \ref {lemma:rX}  to $M=\tilde H^n(\mathring Z)$, we also have a natural map $(\mu_d)_{ab}\otimes H_P(Z')\to H^P(Z)_{\mu_d}$. With the isomorphism $\mu_d\cong \ZZ/d$ coming from
 $\pi_1(\CC^\times, 1)\cong \ZZ$,  we then  get a map 
\[
s_{Z'}: \ZZ/d\otimes H_P(Z')\to H^P(Z)_{\mu_d}.
\]
There is also a simple geometric description of $s_{Z'}$:  an $(n-1)$-cycle $b$ on $Z'$ which represents a  class of $H_P(Z')$ bounds in $Z$: $b=\p a$ for some $(n+1)$-chain $a$ on $Z$. Then $(1-g_0)a$ is an $n$-cycle on $Z$  (recall that $g_0$ is a generator of $\mu_d$) which represents $s_{Z'}[c]\in H^P(Z)_{\mu_d}$.

Now consider  the diagram 
\begin{equation}\label{eqn:basicdiagram}
\begin{tikzcd}
\ZZ/d\otimes H^P(Z') &  \ZZ/d\otimes H_P(Z')\arrow[l, "\ZZ/d\otimes\iota_{Z'}" ']\arrow[d, "s_{Z'}"]\\
H_P(Z)_{\mu_d} \arrow[u, "r_{Z'}"] \arrow[r, "(\iota_Z)_{\mu_d}"]& H^P(Z)_{\mu_d}
\end{tikzcd}
\end{equation}
Since this diagram exists universally over the  $\CC^\times$-bundle over $\mathring\PP_{n,d}$, the group  $\hat\pi_1(\mathring\PP_{n,d},[Z'])$ acts on it. The monodromy along a simple loop as above maps to Picard-Lefschetz transformation in the top and to  its generalized version in the bottom.

Recall that when  $n$ is even, the  top arrow  of diagram \eqref{eqn:basicdiagram} is an isomorphism and that when $n$ is  odd, the  bottom   arrow is.

\begin{theorem}\label{thm:main}
When $n$ is even, $r_{Z'}$ is onto and the  inverse of the top arrow renders the diagram \eqref{eqn:basicdiagram} commutative in the sense that
\[
\begin{tikzcd}
(\iota _Z)_{\mu_d}: H_P(Z)_{\mu_d}\arrow[r, "r_{Z'}", two heads] & \ZZ/d\otimes H^P(Z')\cong  \ZZ/d\otimes H_P(Z')\arrow[r] & H^P(Z)_{\mu_d}.
\end{tikzcd}
\]

When $n$ is odd, $s_{Z'}$ is onto and  the  inverse of the bottom horizontal  arrow renders the diagram commutative  in the sense that  
\[
\begin{tikzcd}
\ZZ/d\otimes\iota_{Z}: \ZZ/d\otimes H_{P}(Z')\arrow[r, two heads, "s_{Z'}"] & H^P(Z)_{\mu_d}\cong H_P(Z)_{\mu_d}\arrow[r ] &
\ZZ/d\otimes H^P(Z').
\end{tikzcd}
\]

In either case, the  kernel of the first surjection  is a subgroup of $\ZZ/d$.
\end{theorem}


It is not clear to us whether in both cases the second map is always  injective,  so that the middle term would be the image
of $(\iota _Z)_{\mu_d}$ resp.\  $\ZZ/d\otimes\iota_{Z}$. 
We shall  however see that this the case when 
$d$ is a prime $p$.  For then $A_p$ is the ring of integers of $\QQ(\mu_p)$ and hence a Dedekind ring. Since $H_P(Z)$ and $H^P(Z)$ are  torsion free, the map $\iota_Z: H_P(Z)\to H^P(Z)$ is an  $A_p$-linear
homomorphism of locally free $A_p$-modules. Their ranks must be  $p_n(p)/(p-1)=p_{n-1}(d)+(-1)^n$.
Since the passage to covariants amounts to applying $\FF_p\otimes_{A_p}$, the induced map
$H_P(Z)_{\mu_p}\to H^P(Z)_{\mu_p}$ is one between  $\FF_p$-vector spaces of the same dimension 
$p_{n-1}(d)+(-1)^n$. A comparison of dimensions then gives:

\begin{corollary}\label{cor:main}
Assume that $d$ is a prime $p$. For $n$ even, $H_P(Z'; \FF_p)\cong H^P(Z'; \FF_p)$ is naturally identified with the image of  
map $(\iota _Z)_{\mu_d}: H_P(Z)_{\mu_d}\to H^P(Z)_{\mu_d}$ and for $n$ odd, $H_P(Z)_{\mu_d}\cong H^P(Z)_{\mu_d}$ is naturally identified with the image of 
$\FF_p\otimes\iota_{Z}:  \FF_p\otimes H_{P}(Z')\to \FF_p\otimes H^{P}(Z')$.\hfill $\square$
\end{corollary}

\begin{example}[Cubic hypersurfaces]
The ring $A_3$ is  the Eisenstein ring. So $H_P(Z)$ and $H^P(Z)$ are free
$A_3$-modules of rank $\half p_n(3)=\frac{1}{3}(2^{n+1}-(-1)^n)$. We get the for $n=1,2,3,4,5,6$ the 
values  $1,2,3,5,11,21,43$ respectively. 
Let us take a closer look  at the two cases  $n=3$ and $n=4$.

When $n=3$, the variety $Z'\subset \PP^3$ is a smooth cubic surface. The primitive  homology $H_P(Z')$ is the part of $H_2(Z')$ annihilated  by canonical class and as is well known, this   lattice is isomorphic with the dual of the root lattice $Q(E_6)$ of type $E_6$, or rather its opposite (the intersection pairing takes the value $-2$ on the roots). The discriminant group of $H_P(Z')$ is cyclic of order 3 with the canonical class providing a generator.  Hence this  intersection pairing becomes  degenerate if we reduce modulo 3 with its radical  being spanned  3 times the canonical class.   
So if we regard $s_{Z'}$  as a map 
\[
s_{Z'}: \FF_3\otimes Q(E_6)\to \FF_3\otimes_{A_3} H^3(Z), 
\]
then by  Corollary \ref{cor:main}, $s_{Z'}$ is an surjection which realizes the nondegenerate quotient of 
$\FF_3\otimes Q(E_6)$. 
This brings us in the context of an exercise in Bourbaki \cite{boubaki_Lie456} (exerc.\ 2 of
Ch.\ 6,  \S 4), which  leads us to the conclusion that the Weyl group $W(E_6)$ acts faithfully on $\FF_3\otimes_{A_3} H^3(Z)$.  In other words,  it is equivalent to impose a principal level 3 structure on $H^3(Z)$ as an $A_3$-module
and to impose a $W(E_6)$-marking of $Z'$. This answers a question asked by Beauville in his overview \cite{beauville} of work of Allcock-Carlson-Toledo \cite{act}  (coming after Prop.\ 8.4) and thus gives a more direct proof that the moduli space of marked cubic surfaces is also  ball quotient (endowed with an action of $W(E_6)$).

When   $n=4$, $Z'\subset \PP^4$ is a smooth cubic threefold and $Z$ is a cubic $4$-fold. Then the image of 
$\FF_3\otimes_{A_3} H_P(Z)\to  \FF_3\otimes_{A_3} H^P(Z)$ is a (hermitian) $\FF_3$-vector space of dimension 10. This vector space  is via $r_{Z'}$ identified with $H^3(Z'; \FF_3)$  and hence 
$H^3(Z'; \FF_3)$ inherits this unitary structure.
\end{example}

\begin{remark}\label{rem:incarnations}
Any divisor  $e$ of $d$  determines  an intermediate $\mu_e$-covering  $Z\to Z_e\to \PP^n$  with 
$Z_e$  the $\mu_{d/e}$-obit space of $Z$. 
It might well be that some version of our main theorem then still holds. 

The most classical case is that of a double cover  $Z_2\to \PP^2$ ramified a smooth plane curve $Z'$ of even degree $d=2e$ and hence of genus $g(Z')=(e-1)(e-2)$. Then the involution acts on the primitive homology $H_P(Z_2)$ as minus the identity so that
its group of co-invariants is simply $H_P(Z_2; \FF_2)$,  whose  dimension is easily determined to be $2g(Z')+1$. The intersection pairing on $H_P(Z_2; \FF_2)$ is degenerate with radical of dimension 1 and an argument as above should  give a natural identification of its nondegenerate quotient with $H_P(Z'; \FF_2)$. This may well be known.

The case when  $Z'$ is a quartic curve in $\PP^2$ was mentioned to me a very long time ago by Bert van Geemen: then  $Z_2$ is a del Pezzo surface of degree 2 and  $H_P(Z_2)$ is up to a sign change of the quadratic form isomorphic with the root lattice $Q(E_7)$ and hence $H_P(Z_2)_{\mu_2}\cong  Q(E_7)\otimes \FF_2$. The quadratic form then defines a map from $ Q(E_7)\otimes \FF_2$ to its dual whose image is via the  above procedure  identified with the 6-dimensional symplectic space $H_1(Z';\FF_2)$. 
Note that this reproduces another exercise in Bourbaki (\cite{boubaki_Lie456}, exerc.\ 3 of Ch.\ 6,  \S 4). 
(To make this connection van Geemen used the double tangents of the quartic interpreted on the genus 3 curve $Z'$ as odd theta characteristics and on $Z$ as `opposite' pairs of exceptional curves.)

Other possible generalizations may regard cyclic covers over  projective manifolds other than a projective space
such as the $\mu_{d/e}$-cover $Z\to Z_e$. Here are two examples which play a central role in Kond\=o's ball quotient description of the moduli spaces of curves of genus 3 \cite{kondo:g=3} and  4 \cite{kondo:g=4}.
They imply that these ball quotient descriptions also exist if we impose a level $2$ structure on a genus 3 curve  or a level 3 structure on a genus 4 curve.

For the genus 3 case, we take the above case: $(d,e)=(4,2)$ and $n=2$. We then get a K3 surface $Z$ of degree 4 appearing  as the double cover of the  degree $2$  del Pezzo surface $Z_2$ ramified along the genus 3  curve $Z'$.  If $H_2(Z/Z_2)$ denotes  the kernel of  $H_2(Z)\to H_2(Z_2)$, then  $H_2(Z/Z_2)$ is a module of the Gaussian integers $\ZZ[\sqrt{-1}]$. It  is free rank 7. The intersection pairing becomes degenerate on
$\FF_2\otimes_{\ZZ[\sqrt{-1}]} H_P(Z/Z_2)$ (with one-dimensional radical) and its nondegenerate quotient can be identified with $H_1(Z'; \FF_2)$.

For the genus 4 case we consider  a $\mu_3$-cover $Z\to \PP^1\times\PP^1$ ramified along a smooth curve $Z'$ of bidegree $(3,3)$. Such a $Z'$  is a general  curve of genus 4 and $Z$ is  a K3 surface of degree $6$. Then the kernel $H_2(Z/(\PP^1\times \PP^1))$ of $H_2(Z)\to H_2(\PP^1\times \PP^1)$ 
is a free $A_3$-module of rank 10 for which the intersection pairing on 
$\FF_3\otimes_{A_3} H_P(Z/\PP^1\times \PP^1)$ is degenerate  with a $2$-dimensional radical 
and whose  nondegenerate quotient can be identified with  $H_1(Z'; \FF_3)$.
\end{remark}

\section{Proof of Theorem \ref{thm:main}}
The statement is independent of the particular choice of $Z'$ and hence we may (and will) assume that
$Z'$ is a Fermat hypersurface, i.e., admits an equation $F=\sum^n_{i=0} z_i^d$. (The homotopy 
type of $\mathring Z$ is that of the join of $\mathring Z'$ and a principal (discrete) $\mu_d$-set and although the $\mu_d$-action on $\mathring Z$ is compatible with this homotopy equivalence, the $\mu_d$-action on $\mathring Z'$ is nontrivial  and the Fermat hypersurface has the advantage of making  this action explicit.)

So if $f: \CC^{n+1}\to \CC$ is the homogeneous function defined by $F$, then  $f^{-1}(0)$ 
is the affine cone over $Z'$ and $\mathring Z=Z\ssm Z'$ is identified with $f^{-1}(-1)$. 
Furthermore,    $\mu_d$ acts by scalar  multiplication in $\CC^{n+1}$. 
By duality, $H^n(Z,Z')\cong H_n(\mathring Z)=H_n(f^{-1}(-1))$ and according to  
Pham \cite{pham} (see also Milnor \cite{milnor}), we have an isomorphism of $\ZZ\mu_d$-modules 
\[
\tilde H_n(f^{-1}(1))\cong \underbrace{I_d\otimes_\ZZ\cdots\otimes_\ZZ  I_d}_{n+1}
\]
(with $n+1$ tensor factors; we shall write this as $I_d^{(n+1)}$ ), where 
$g\in \mu_d$ acts diagonally as $g\otimes\cdots \otimes g$.  Note that this implies the earlier assertion that
$H_n(\mathring Z)$ is free abelian of rank $(d-1)^{n+1}$.

We get from Lemma \ref{lemma:hyperplanecut} the  short exact sequence
\begin{gather}
0 \to H^P(Z')\to I_d^{(n+1)}\to H_P(Z)\to 0, \label{eqn:ses2}
\end{gather}
of  $\ZZ\mu_d$-modules with  $\mu_d$ acting  trivially on the (co)-homology of $Z'$.
\medskip

The generator $g_0$ of $\mu_d$  identifies the group ring $R_d$ with $\ZZ[x]/(x^d-1)$,  the augmentation ideal $I_d$ with the ideal generated by $x-1$.

Hence $I_d^{(n+1)}\subset R_d^{(n+1)}$ is the ideal in $\ZZ[x_0, \dots, x_n]$ (with $x_i^d=1$ for all $i$) generated by $(x_0-1)\cdots (x_n-1)$ and the generator $g\in {\mu_d}$ acts on $\ZZ[x_0, \dots, x_n]$  and on this ideal as multiplication with $x_0\cdots x_n$. This suggests to pass to the generators $y_k:=x_0x_1\cdots x_k$, $k=0, \dots, n$ (which then will still be subject to the relations $y_k^d=1$). 
Then  $R_d^{(n+1)}= R_d^{(n)}\otimes \ZZ[y_n]/(y_n^d-1)$.
Since   $x_k=y_k/y_{k-1}$ for $k=1, \dots , n$, the ideal $I_d^{(n+1)}\subset R_d^{(n+1)}$  is  generated by 
\[
\varphi_n:=(y_0-1) (y_1-y_0)\cdots (y_n-y_{n-1})\in R_d^{(n+1)}
\]
and  $g\in {\mu_d}$ acts  as multiplication with $y_n$. We denote the image of $1\in R_d^{(n+1)}$ under $\eps^d(1)=\sum_{k\in \ZZ/d} g^k$ by 
\[
\textstyle u_n:=\sum_{k\in \ZZ/d} y_n^k.
\]
Note that $R_d^{(n+1)}u_n$ and  $R_d^{(n+1)}(y_n-1)$ are each others annihilator. Multiplication with $u_n$ has in $R_d^{(n+1)}$ the effect of substituting $1$ for  $y_n$ and  so $R_d^{(n+1)}u_n=R_d^{(n)}u_n$. 
This ideal is also the group of $\mu_d$-invariants in $R_d^{(n+1)}$.

\begin{lemma}\label{lemma:inv}
The group of $\mu_d$-invariants in $I_d^{(n+1)}=R_d^{(n+1)} \varphi_{n}$ is the ideal 
$\big(R^{(n)}_d\varphi_{n-1} \cap  R^{(n)}_d(1-y_{n-1})\big)u_{n}$
in $R^{(n+1)}_d$. 
\end{lemma} 
\begin{proof}
The group of $\mu_d$-invariants in $I_d^{(n+1)}$ is $I_d^{(n+1)}\cap R_d^{(n+1)}u_n$.  To show that this intersection is as asserted, let  $\varphi_na$ be any element of $I_d^{(n+1)}$ and write $a$ as 
$\sum_{k\in \ZZ/d} a_k y_n^k$ with $a_k\in R_d^{(n)}$. Then 
\[
\textstyle \varphi_n a=\varphi_{n-1} (y_n-y_{n-1})\sum_{k\in \ZZ/d} a_k y_n^k=
\varphi_{n-1}\sum_{k\in\ZZ/d} (a_{k-1}-y_{n-1}a_k)y_n^k.
\]
This is ${\mu_d}$-invariant, i.e.,  of the form $\varphi_{n-1} a'u_n$ with $a'\in R_d^{(n)}$,  
 precisely when $\varphi_{n-1}(a_{k-1}-y_{n-1}a_k)=\varphi_{n-1} a'\in R_d^{(n)}$ for all $k\in \ZZ/d$. This means that 
\[
\varphi_{n-1} a_k= \varphi_{n-1} a'(y_{n-1}^{-1}+\cdots +y_{n-1}^{-k}) + \varphi_{n-1} a_0 y_{n-1}^{-k}
\]
Here $a_0$ can be arbitrary, but  $a'$ is subject to a condition, for  by taking  $k=d$, we see that 
 $\varphi_{n-1} a'(y_{n-1}^{-1}+\cdots +y_{n-1}^{-d})=\varphi_{n-1} a'u_{n-1}$  must vanish. This amounts to demanding that $\varphi_{n-1} a'$ is divisible by $1-y_{n-1}$.
\end{proof}
 
So Lemma \ref{lemma:rX}  then  gives us  an $R^{(n)}_d$-linear isomorphism
\begin{equation}\label{eqn:1}
H^P(Z')\cong R^{(n)}_d\varphi_{n-1}\cap R^{(n)}_d(1-y_{n-1})
\end{equation}
and  an $R_d^{(n+1)}$-linear isomorphism
\begin{equation}\label{eqn:2}
H_P(Z)\cong R_d^{(n+1)} \varphi_{n} /\big(R_d^{(n)}\varphi_{n-1}\cap R_d^{(n)}(1-y_{n-1}) \big)u_{n}.
\end{equation}

The following corollary corrects corollary 2.2  in our paper \cite{looijenga}.  This error  was brought to our attention by D.~Gvirtz (see section 2 and in particular Remark 2.2  in  \cite{gvirtz}).

\begin{corollary}\label{cor:compare}
The natural $R_d^{(n)}$-linear map $\iota_Z: H_P(Z)\to  H^P(Z)$ is identified with the map 
\begin{equation}\label{eqn:basic}
R_d^{(n+1)}\varphi_{n} /\big(R_d^{(n)}\varphi_{n-1}\cap R_d^{(n)}(1-y_{n-1})\big)u_{n}\to 
R^{(n+1)}_d\varphi_{n} \cap  R^{(n+1)}_d(1-y_{n}).
\end{equation}
induced by multiplication with $1-y_{n}$. In particular, this map is an isomorphism when $n$ is odd and  its kernel and cokernel are naturally isomorphic with $\ZZ/d$ when $n$  is even.
\end{corollary}
\begin{proof}
The identifications of $H_P(Z)$ and $H^P(Z)$ are provided by the isomorphisms \eqref{eqn:1} and 
\eqref{eqn:2}  (where in the last  case we replace $n-1$ by $n$).
The natural map $\tilde H_n(\tilde Z)\to \tilde H^n(\tilde Z)$ is given by 
multiplication with $1-y_{n}$ (see Pham \cite{pham} \S 3) and so this induces also the map $H_P(Z)\to H^P(Z)$ (note that the map \eqref{eqn:basic} is well-defined as this kills $u_{n}$).
\end{proof}

\begin{remark}\label{rem:compare}
So Corollary \ref{cor:compare} tells us that for odd $n$, the ideals 
$R^{(n+1)}_d\varphi_{n}$ and $R^{(n+1)}_d(1-y_{n})$ are relatively prime in the sense that their intersection is their product $R_d^{(n+1)}\varphi_{n}(1-y_{n})$. It also tells us that  the annihilator of  $1-y_{n}$ in $R^{(n+1)}_d\varphi_{n}$ (which is $R^{(n+1)}_d\varphi_{n}\cap  R^{(n+1)}_d u_n$) is then equal to $\big(R_d^{(n)}\varphi_{n-1}\cap R_d^{(n)}(1-y_{n-1})\big)u_{n}$
\end{remark}

\begin{proof}[Proof of Theorem \ref{thm:main}]
We do this by writing the diagram  \eqref{eqn:basicdiagram} out in the above terms. It
is induced by the diagram  of $R^{(n)}_d$-modules in which the sources (right top and left bottom) 
represent $H_P(Z')$ and $H_P(Z)$ respectively

\begin{large}
\[
\begin{tikzcd}
\frac{R^{(n)}_{d}\varphi_{n-1}\cap R^{(n)}_{d}(1-y_{n-1})}{d(R^{(n)}_{d}\varphi_{n-1}\cap R^{(n)}_{d}(1-y_{n-1}))}
 &  \frac{R^{(n)}_{d}\varphi_{n-1}}{\big(R^{(n-1)}_{d}\varphi_{n-2}+R^{(n-1)}_{d}(1-y_{n-2})\big)u_{n-1}} \arrow[l, ".(1-y_{n-1})" ']\arrow[d, ".(y_{n}-y_{n-1})(1-y_n)"]\\
\frac{R^{(n+1)}_{d}\varphi_{n}}{\big(R^{(n)}_{d}\varphi_{n-1}\cap R^{(n)}_{d}(1-y_{n-1})\big)u_{n}}\arrow[u, "y_n=1"] \arrow[r, ".(1-y_n)"]& \frac{R^{(n+1)}_{d}\varphi_{n}\cap R^{(n+1)}_{d}(1-y_n)}{\big(R^{(n+1)}_{d}\varphi_{n}\cap R^{(n+1)}_{d}(1-y_n)\big)(1-y_n)}
\end{tikzcd}
\]
\end{large}
The identity $u_{n-1}(1-y_{n-1})=0$  resp.\  $u_n(1-y_n)=0$ ensures that the bottom and top horizontal maps are well-defined. The map on the left is well-defined because $y_n=1$ takes $u_n$ to $d$ and the map on the right is well-defined because $u_{n-1}(y_n-y_{n-1})=u_{n-1}(1-y_{n-1})$. We then observe that in this diagram 
\[
\begin{tikzcd}
\varphi_{n-1}(1-y_{n-1})&  \varphi_{n-1}\arrow[l, maps to, ".(1-y_{n-1})" ']\arrow[d, maps to, ".(y_{n-1}-y_{n-2})(1-y_n)"]\\
\varphi_n\arrow[u, maps to, "y_n=1"] \arrow[r, maps to, ".(1-y_n)"]& \varphi_n(1-y_n)
\end{tikzcd}
\]
This gives us, combined with Remark \ref{rem:compare}, all the asserted properties. 
\end{proof}


\end{document}